\documentclass[psamsfonts]{amsart} 
\usepackage{amsmath,amsfonts,amssymb,amsthm}
\usepackage{latexsym}
\usepackage{enumerate}
\usepackage{framed}
\usepackage{psfrag}
\usepackage{graphicx}
\usepackage{subfig}
\usepackage{algorithm}
\usepackage[noend]{algorithmic}

\newcommand{\N}{\mathbb{N}}
\newcommand{\Z}{\mathbb{Z}}
\newcommand{\Q}{\mathbb{Q}}
\newcommand{\R}{\mathbb{R}}
\newcommand{\C}{\mathbb{C}}

\newcommand{\calE}{\mathcal{E}}
\newcommand{\calO}{\mathcal{O}}

\newcommand{\fraka} {\mathfrak{a}}

\newcommand{\abs}[1]{{\left|{#1}\right|}}
\newcommand{\norm}[1]{{\left\|{#1}\right\|}}
\newcommand{\ggen}[1]{{\left\langle{#1}\right\rangle}}

\DeclareMathOperator{\vol}{vol}
\DeclareMathOperator{\Id}{Id}
\DeclareMathOperator{\mymod}{mod}
\DeclareMathOperator{\Red}{Red}

\newenvironment{enuma}{\begin{enumerate}[\upshape (a)]}{\end{enumerate}}

\newtheorem{theorem}{Theorem}
\newtheorem{corollary}{Corollary}
\newtheorem{lemma}[theorem]{Lemma}
\newtheorem{proposition}{Proposition}

\theoremstyle{definition}

\title{Rigorous Computation of Fundamental Units in Algebraic Number Fields}

\author{Felix Fontein}
\address{Department of Mathematics \&\ Statistics, University of Calgary, 2500 University Drive NW, Calgary, Alberta, Canada T2N 1N4}
\email{fwfontei@ucalgary.ca}

\author{Michael J. Jacobson, Jr.}
\address{Department of Computer Science, University of Calgary, 2500 University Drive NW, Calgary, Alberta, Canada T2N 1N4}
\email{jacobs@ucalgary.ca}
\thanks{The second author is supported in part by NSERC of Canada.}

\begin{document}
  
\maketitle
  
\begin{abstract}
We present an algorithm that unconditionally computes a representation
of the unit group of a number field of discriminant $\Delta_K,$ given
a full-rank subgroup as input, in asymptotically fewer bit operations
than the baby-step giant-step algorithm.  If the input is assumed to
represent the full unit group, for example, under the assumption of
the Generalized Riemann Hypothesis, then our algorithm can
unconditionally certify its correctness in expected time
$O(\Delta_K^{n/(4n + 2) + \epsilon}) = O(\Delta_K^{1/4 - 1/(8n+4) +
\epsilon})$ where $n$ is the unit rank.
\end{abstract}

\section{Introduction}
  
Let $K$ be an algebraic number field of discriminant $\Delta_K.$ One
of the main computational problems in algebraic number theory is to
compute a representation of the group of units of the corresponding
maximal order $\calO_K$. The units are of interest in a number of
contexts.  As an example, it is well-known that computing the
fundamental unit of a real quadratic field is equivalent to solving
the Pell equation $x^2 - D y^2 = 1.$

In general, the unit group consists of a finite torsion subgroup and
an infinite part of rank $n,$ where $n$ is called the unit rank.  A
generating system of the infinite part is called a system of
fundamental units.  The torsion subgroup is an easily-computed group
of roots of unity, so computing the unit group means determining a
system of fundamental units.  Instead of directly computing the units
themselves, many algorithms compute a basis of the corresponding
logarithm lattice $\Lambda_K,$ a rank $n$ lattice in $\R^n$ derived
from the Archimedean absolute values of $K.$ The fundamental units can
be recovered from a basis of $\Lambda_K$ (see, for example,
\cite{thiel-comprep}).

The fastest algorithms for unconditionally computing a system of
fundamental units, meaning that they generate the entire unit group
without having to rely on any unproved assumptions or heuristics, are
of exponential complexity in the bit length of the field discriminant.
The current state-of-the-art is due to Buchmann \cite{buchmann-habil},
whose algorithm computes a basis of the logarithm lattice in
$O(\Delta_K^{1/4 + \epsilon})$ bit operations\footnote{Throughout this
paper, the $O$-constants are assumed to be dependent on the degree $[K
~:~ \Q]$ of $K.$ Furthermore, to simplify notation, expressions
involving $\Delta_K$ should be assumed to operate on $|\Delta_K|.$}.
However, if one is willing to assume the truth of the Generalized
Riemann Hypothesis (GRH), then Buchmann's index-calculus algorithm
\cite{buchmann-subexp-regulator} can be used.  This algorithm has
subexponential complexity in $\log \Delta_K$ assuming the GRH, but
unfortunately the correctness of the output also depends on the GRH.

The motivating question for the work in this paper is whether it is
possible to certify that the logarithm lattice of a unit group
produced by the index-calculus algorithm is unconditionally correct in
asymptotically fewer than $O(\Delta_K^{1/4+\epsilon})$ bit operations.
More generally, given a full rank sublattice $\Lambda'$ of the
logarithm lattice corresponding to the unit group of a number field
$K,$ is it possible to compute the full logarithm lattice in fewer
than $O((\det \Lambda')^{1/2 + \epsilon} \Delta_K^\epsilon)$ bit
operations, i.e., faster than using baby-step giant-step?

These questions were answered affirmatively for the case of real
quadratic fields in \cite{haan-jacobson-williams-fastrigorous}.  The
unit group of a real quadratic field of discriminant $\Delta$ has rank
one, generated by a single fundamental unit $\varepsilon_\Delta >
1$. The corresponding lattice of logarithms is generated by a single
real number, the regulator $R_\Delta = \log \varepsilon_\Delta.$ In
\cite{haan-jacobson-williams-fastrigorous}, it is proved that an
unconditionally correct approximation of $R_\Delta$ can be computed in
time $O(S^{1/3} \Delta^\epsilon)$ given an integer multiple $S$ of
$R_\Delta.$ Furthermore, if it is assumed that $S$ is the output of
the index-calculus algorithm, then, assuming the GRH, $S$ is the
regulator and hence of size $O(\Delta^{1/2 + \epsilon}).$ The end
result is an algorithm that unconditionally computes the regulator in
expected time $O(\Delta^{1/6 + \epsilon})$ assuming the GRH.  This
algorithm was shown to work very well in practice, as demonstrated by
the computation of the regulator of a real quadratic field with
$65$-decimal digit discriminant, the largest such result to-date.

In this paper, we generalize this result to computing a basis of the
logarithm lattice corresponding to the unit group of an algebraic
number field $K$ with arbitrary unit rank, given a full rank
sublattice $\Lambda'$ as input.  In particular, we describe an
algorithm that solves this problem in $O((\det \Lambda')^{n/(2n + 1) +
\epsilon} \Delta_K^\epsilon)$ bit operations.  For unit rank one
fields we recover the same complexity as
\cite{haan-jacobson-williams-fastrigorous}, and the algorithm is
asymptotically faster than $O((\det \Lambda')^{1/2 + \epsilon}
\Delta_K^\epsilon)$ for all $n.$ When $\Lambda'$ is computed using
the index-calculus algorithm, we have, similar to the quadratic case,
that it is in fact the full logarithm lattice under the assumption of
the GRH.  Thus, we obtain an algorithm for computing the logarithm
lattice unconditionally in expected $O(\Delta_K^{n / (4n + 2) +
\epsilon})$ bit operations assuming the GRH. Our algorithm is
asymptotically faster than $O(\Delta_K^{1/4 + \epsilon})$ for all
$n,$ but the greatest improvements occur for small $n.$ For example,
for fields of unit rank one we obtain $O(\Delta_K^{1/6 +
\epsilon}),$ the same complexity as
\cite{haan-jacobson-williams-fastrigorous} in the real quadratic case,
and for unit rank two we obtain $O(\Delta_K^{1/5 + \epsilon}).$

The paper is organized as follows.  Following a presentation of the
required notation and background in Section~\ref{sec:notation}, we
give an overview of the algorithm in Section~\ref{sec:overview}.  The
theory behind the algorithm is described in detail in
Section~\ref{sec:latticemax}, and two important subroutines are
described in Section~\ref{sec:bsgssearch}.  The algorithm itself and a
proof of its complexity are given in Section~\ref{sec:thealgsect}, and
we finish with some concluding remarks.

\section{Notation and Background}
\label{sec:notation}
  
All required information on number fields can be found in
\cite{neukirch}.  References are provided for results not appearing in
this source.
  
Let $K$ be a number field, i.e. a finite extension of $\Q$. Denote the integral closure of $\Z$ in
$K$ by $\calO_K$. This is a Dedekind domain. Let $\abs{\bullet}_1, \dots, \abs{\bullet}_{n+1}$ be
all $n + 1$ Archimedean absolute values of $K$; these correspond to embeddings $\sigma_i : K \to \C$
up to complex conjugation by $\abs{f}_i = \abs{\sigma_i(f)}$, $f \in K$, $1 \le i \le n+1$. Let
$\deg \abs{\bullet}_i := 1$ if $\sigma_i(K) \subseteq \R$, and $\deg \abs{\bullet}_i := 2$
otherwise. Consider the map \[ \Psi : K^* \to \R^n, \qquad f \mapsto (\log \abs{f}_1, \dots, \log
\abs{f}_n). \] The image of the unit group~$\calO_K^*$ is a lattice of rank~$n$, denoted by
$\Lambda_K := \Psi(\calO_K^*)$. The kernel of $\Psi|_{\calO_K^*} : \calO_K^* \to \Lambda_K$ is the
group of roots of unity in $K$, $\mu_K$, and we have that $\calO_K^* \cong \mu_K \times \Lambda_K
\cong \mu_K \times \Z^n,$ where the number~$n$ is called the \emph{unit rank} of $K$.  Thus, every
unit in $\calO_K^*$ can be written as $\zeta \varepsilon_1^{k_1} \dots \varepsilon_n^{k_n},$ where
$\zeta \in \mu_K$ and $\varepsilon_1, \dots, \varepsilon_n$ are a system of \emph{fundamental units}
of $\calO_K^*.$ The \emph{regulator} $R_K$ of $K$ equals $\det \Lambda_K \cdot \prod_{i=1}^n \deg
\abs{\bullet}_i$.

One can recover a unit~$\varepsilon$ from its image
$\Psi(\varepsilon)$ up to a root of unity.  If one sets $t_i := \log
\abs{\varepsilon}_i$, $1 \le i \le n$ and $t_{n+1} := -\frac{1}{\deg
\abs{\bullet}_{n+1}} \sum_{i=1}^n t_i \deg \abs{\bullet}_i$, one has
that $\mu_K \varepsilon \cup \{ 0 \} = \{ f \in \calO_K \mid \log
\abs{f}_i \le t_i \text{ for } 1 \le i \le n + 1 \}$.  Thus, computing
a basis of $\Lambda_K$ allows us to recover a system of fundamental
units, thereby completely determining the unit group of $\calO_K.$
  
Another important invariant of $K$ is the
\emph{discriminant}~$\Delta_K$; it is defined as follows. The
ring~$\calO_K$ is a free $\Z$-module of rank~$d = [K : \Q]$; let $v_1,
\dots, v_d \in \calO_K$ be a $\Z$-basis of $\calO_K$. Moreover, as $K
/ \Q$ is separable, one has $d$ distinct embeddings $\sigma_1, \dots,
\sigma_{n+1}, \sigma_{n+2}, \dots, \sigma_d : K \to \C$. The
discriminant~$\Delta_K$ is defined as $\det(A)^2$, where $A =
(\sigma_i(v_j))_{1 \le i, j \le d} \in \C^{n \times n}$; it can be
shown that $\Delta_K \in \Z \setminus \{ 0 \}$, with $\Delta_K \neq 
\pm 1$ for $K \neq \Q$. In order to simplify the notation, $\Delta_K$ 
should be understood to be in absolute value when required in 
arithmetic expressions and complexity statements.
  
Let $g : \R^{t+1} \to \R_{>0}$ be a function and $x_1, \dots, x_t$ be parameters which can depend on
the number field $K$; examples are $\Delta_K$, $R_K$ and $n$. We say that a quantity~$f(x_1, \dots,
x_t)$ is in $O(g(x_1, \dots, x_n, \epsilon))$, if there exist a family of constants~$C_{[K :
\Q],\epsilon} > 0$, only depending on $[K : \Q]$ and $\epsilon$, such that for all $\epsilon > 0$
and all number fields~$K$, $f(x_1(K), \dots, x_n(K)) \le C_{[K : \Q],\epsilon} \cdot g(x_1(K),
\dots, x_n(K), \epsilon)$ for sufficiently large $x_1(K), \dots, x_n(K).$ In that case, we write $f
= O(g(x_1, \dots, x_n, \epsilon))$. This simply means that the $O$-constant depends only on the
extension degree~$[K : \Q]$, and not on any other information of $K$ or any other parameter.
  
In the following, we will use that $R_K = O(\Delta_K^{1/2 +
\epsilon})$ by a result of Sands \cite{sands-generalization}, as
well as that $\det \Lambda_K = R_K / \prod_{i=1}^n \deg \abs{\bullet}_i$
can be bounded from below only in terms of $[K : \Q]$ by a result of
Remak \cite{remak-regulator}. The latter means that for any
sublattice~$\Lambda' \subseteq \Lambda_K$, we have $[\Lambda_K :
\Lambda'] = \det \Lambda' / \det \Lambda_K = O(\det
\Lambda')$. Moreover, we will use that arithmetic in $K$ can be done
in $O(\Delta_K^\epsilon)$~bit operations; see, for example,
\cite{genlagrange,buchmann-habil}.
  
Finally, for $v \in \R^n$ and $M \subseteq \R$, we set $M v := \{ v m \mid m \in M \}$, and for
subsets~$M', M'' \subseteq \R^n$, we set $M' + M'' := \{ m' + m'' \mid m' \in M', m'' \in M''
\}$. We equip $\R^n$ with the Euclidean norm, denoted by $\norm{\bullet}$, as well as with the
Lebesgue measure, denoted by $\vol$.

\section{Overview of the Algorithm}
\label{sec:overview}
  
Our algorithm will, given a sublattice~$\Lambda' \subseteq \Lambda_K$
of full rank~$n$, compute $\Lambda_K$ in $O((\det
\Lambda')^{\frac{n}{2n + 1} + \epsilon} \Delta_K^\epsilon)$~bit
operations, using $O((\det \Lambda')^{\frac{n}{2 n + 1}}
\Delta_K^\epsilon)$~bits of storage.
  
The idea can be sketched as follows. Since $\Lambda'$ is of full rank,
the quotient group~$\Lambda_K / \Lambda'$ is finite. Denote its order
by $i_{\Lambda'}$. Now we do not know $\Lambda_K$ or $i_{\Lambda'}$,
but there is an effective test whether a prime~$p$ divides the
index~$i_{\Lambda'}$ based on the following proposition, which we will
prove in Section~\ref{sec:latticemax}.
  
\begin{proposition}
  \label{mainpropoverview}
  Assume that $\Lambda' = \sum_{i=1}^n \Z v_i$ for a basis~$(v_1,
  \dots, v_n)$ of $\R^n$. A prime~$p$ divides~$i_{\Lambda'}$ if, and
  only if, there is an element of $\Lambda_K$ in \[ \bigcup_{k=1}^n
  \biggl\{ \sum_{i=1}^{k-1} \frac{a_i}{p} v_i + \frac{1}{p} v_k
  \;\biggm|\; a_1, \dots, a_{k-1} \in \{ 0, \dots, p - 1 \}
  \biggr\}. \] If such an element~$v$ exists, set $\Lambda'' :=
  \Lambda' + \Z v$. This is a sublattice of $\Lambda_K$ with
  $i_{\Lambda''} = [\Lambda_K : \Lambda''] =
  \frac{i_{\Lambda'}}{p}$. \qed
\end{proposition}
  
\begin{figure}[b]
  \psfrag{v1}{$v_2$}
  \psfrag{v2}{$v_1$}
  \psfrag{V}{$V_B$}
  \psfrag{p=2}{$p = 2$}
  \psfrag{p=3}{$p = 3$}
  \psfrag{p=5}{$p = 5$}
  \psfrag{k1}{$k = 1$}
  \psfrag{k2}{$k = 2$}
  \begin{center}
    \subfloat[Proposition~\ref{mainpropoverview} in case $n = 2$ and $p = 5$]%
    {\includegraphics[height=3cm]{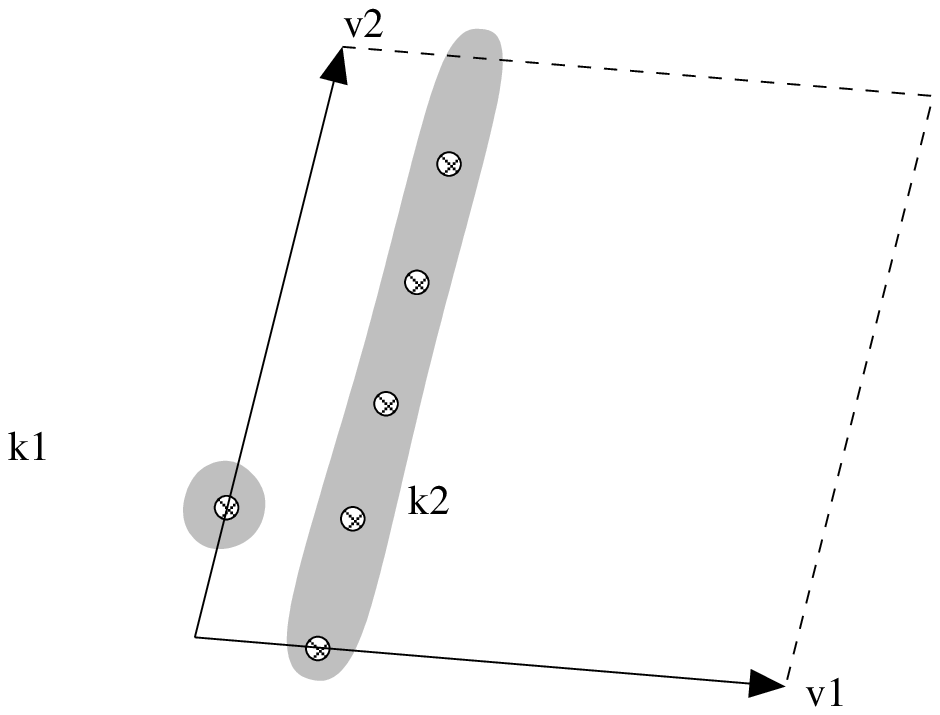}\label{fig:propsketch}}
    \qquad\qquad
    \subfloat[Overview of the algorithm for $n = 2$ and $B = 6$]%
    {\includegraphics[height=3cm]{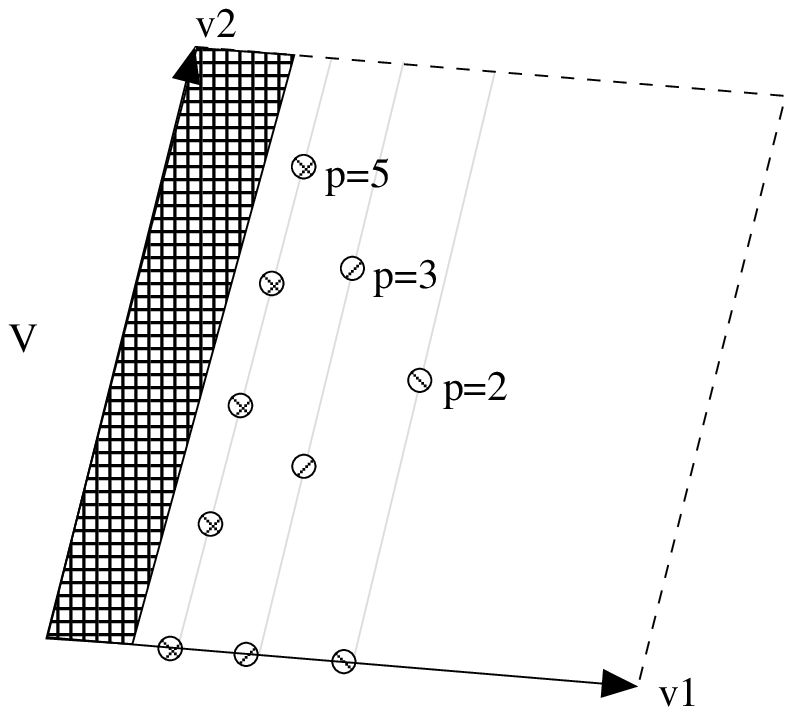}\label{fig:algsketch}}
  \end{center}
  \caption{Overview of the Algorithm}
\end{figure}

The search set in the proposition is shown in
Figure~\ref{fig:propsketch}. If we would have a finite set of
candidates for prime divisors of $i_{\Lambda'}$, we could iterate
through the set of candidates and use the proposition to determine the
prime divisors of $i_{\Lambda'}$, their multiplicities and, most
importantly, $\Lambda_K$ itself. Unfortunately, as $i_{\Lambda'} =
\calO(\det \Lambda'),$ this method would in general be slower than
baby-step giant-step.
  
Alternatively, one could simply search a fundamental parallelepiped of
$\Lambda'$, such as $\sum_{i=1}^n [0, 1] v_i$, for elements of
$\Lambda_K$. Using Buchmann's baby-step giant-step method for number
fields as presented in \cite{buchmann-habil}, this can be done in
$O((\det \Lambda')^{1/2} \Delta_K^\epsilon)$~bit operations. But
instead, one could also directly apply Buchmann's method to compute a
basis for $\Lambda_K$ and compare it to $\Lambda'$; if $i_{\Lambda'}
\gg 1$, this would actually be faster.
  
The idea of our algorithm is to combine both approaches. First, we
test all primes~$p$ below a bound~$B$ using an algorithm based on
Proposition~\ref{mainpropoverview}. After that, we use Buchmann's
algorithm to search a small subset of the fundamental parallelepiped
for elements of $\Lambda_K$.  Note that the set of elements we have to
search for Proposition~\ref{mainpropoverview} lies in a small subset
of the fundamental parallelepiped, as illustrated in
Figure~\ref{fig:propsketch}. More precisely, if $\Lambda' =
\sum_{i=1}^n \Z v_i$ as in the proposition, the search set for a
prime~$p$ lies in \[ V_p := \sum_{i=1}^{n-1} [0, 1] v_i + [0,
\tfrac{1}{p}] v_n. \] Moreover, if $q \ge p$, then $V_q \subseteq
V_p$. Therefore, if we use the method from
Proposition~\ref{mainpropoverview} for all primes~$\le B$, then it
suffices to search the set $V_B$ using Buchmann's method, as
illustrated in Figure~\ref{fig:algsketch}.  Finding an optimal value
of $B$ that minimizes the total running time of the two parts of the
algorithm gives us the results.
  
Note that we ignore all approximation issues in this algorithm, and
refer to the discussion in Sections~13 and 16 of
\cite{buchmann-habil}.

\section{Lattice Maximization}
\label{sec:latticemax}
  
Lattice maximization refers to the process described in the previous
section, proving that $\Lambda' = \Lambda_K$ or finding a sublattice
$\Lambda''$ with $\Lambda' \subset \Lambda'' \subseteq \Lambda_K.$ In
this section, we describe in more detail the lattice maximization
algorithm outlined in the previous section, and prove the results
required to establish its correctness and complexity.

We begin with a lemma which allows us to determine whether an integer
is coprime to the index~$i_{\Lambda'}$.
\begin{lemma}
  \label{indexcoprimetestlemma}
  An integer~$t > 0$ has a non-trivial common divisor with $i_{\Lambda'}$ if, and only if, $\Lambda'
  \subsetneqq \frac{1}{t} \Lambda' \cap \Lambda_K$.  Moreover, any element~$v \in (\frac{1}{t}
  \Lambda' \cap \Lambda_K) \setminus \Lambda'$ gives rise to a sublattice~$\Lambda'' := \Lambda' +
  \Z v \supsetneqq \Lambda'$ with $\Lambda'' \subseteq \Lambda_K$, and $[\Lambda'' : \Lambda']$ is a
  divisor of $t$.
\end{lemma}
  
\begin{proof}
  First, assume that $d = \gcd(t, i_{\Lambda'}) > 1$. Let $p$ be a
  prime dividing~$d$. Then there exists an element~$v \in \Lambda_K
  \setminus \Lambda'$ with $p v \in \Lambda'$. But then, $v \in
  (\Lambda_K \cap \frac{1}{p} \Lambda') \setminus \Lambda' \subseteq
  (\Lambda_K \cap \frac{1}{t} \Lambda') \setminus \Lambda'$.
    
  On the contrary, assume that there exists some~$v \in (\Lambda_K
  \cap \frac{1}{t} \Lambda') \setminus \Lambda'$. Then $t v \in
  \Lambda'$, whence the order of $v$ in $\Lambda' / \frac{1}{t}
  \Lambda'$ is a non-trivial divisor of $t$. But since the order
  divides $\abs{\Lambda' / \frac{1}{t} \Lambda'} = i_{\Lambda'}$, we
  see that $\gcd(t, i_{\Lambda'}) > 1$.
    
  For any~$v \in (\frac{1}{t} \Lambda' \cap \Lambda_K) \setminus
  \Lambda'$, we have $\Lambda'' := \Lambda' + \Z v \subseteq
  \Lambda_K$ and $\Lambda' \subsetneqq \Lambda''$, and since $t v \in
  \Lambda'$ we see that $\Lambda'' / \Lambda' = (v + \Lambda') \Z$ is
  cyclic of order dividing $t$. \qed
\end{proof}
  
Note that we have a tower of subgroups $\Lambda' \subseteq
\tfrac{1}{t} \Lambda' \cap \Lambda_K \subseteq \tfrac{1}{t} \Lambda'.$
The lemma says that $\gcd(t, i_{\Lambda'}) > 1$ if, and only if,
$(\frac{1}{t} \Lambda' \cap \Lambda_K) / \Lambda'$ is not the trivial
subgroup of $\frac{1}{t} \Lambda' / \Lambda'$. Hence, if we let~$p$ be
a prime divisor of $i_{\Lambda'}$, we can replace~$\Lambda'$ by a
sublattice~$\Lambda''$ with $i_{\Lambda''} = \frac{i_{\Lambda'}}{p}$
by searching a set of representatives of $\Lambda' / \frac{1}{p}
\Lambda'$. But this can be done more efficiently, as hinted in
Proposition~\ref{mainpropoverview}. This is provided by the following
result; note that $\frac{1}{t} \Lambda' / \Lambda' \cong (\Z/t\Z)^n$.
  
\begin{proposition}
  \label{finitegrouptrivialdetectlemma}
  Let $G$ be a finite group, and let $H \subseteq G$ be a
  subgroup. Let $S$ be the set of all cyclic subgroups of prime order
  of $G$, and let $\tilde{S} \subseteq G$ such that for every~$U \in
  S$, there exists a unique element~$g \in \tilde{S}$ with $U =
  \ggen{g}$.
  \begin{enuma}
    \item The subgroup $H$ of $G$ is trivial if, and only if, $\tilde{S} \cap H = \emptyset$.
    \item If $G = (\Z/p\Z)^m$ for a prime~$p$ and $m \in \N$, we
      can choose the set~$\tilde{S}$ to be a subset of \[ \{ (a_1,
      \dots, a_m) + p \Z^m \mid (a_1, \dots, a_m) \in \{ 0, \dots, p-1
      \}^m, \; a_m \le 1 \}. \]
  \end{enuma}
\end{proposition}
  
\begin{proof}\hfill
  \begin{enuma}
    \item The neutral element~$e$ generates the trivial subgroup of
      $G$. Hence, if $H = \{ e \}$, then $H \cap \tilde{S} =
      \emptyset$. Conversely, assume that $\abs{H} > 1$. Then there
      exists an element~$g \in H$ of prime order, and $\ggen{g}$ is a
      non-trivial cyclic subgroup of prime order of $G$. Hence,
      $\ggen{g} \in S$, and there exists some~$\tilde{g} \in
      \tilde{S}$ with $\ggen{g} = \ggen{\tilde{g}}$. In particular,
      $\tilde{g} \in \ggen{g} \subseteq H$, whence $\tilde{S} \cap H
      \neq \emptyset$.
    \item Let $v = (v_1, \dots, v_m) + p \Z^m \in G$. If $p \mid v_m$,
      set $\lambda := 1$. Otherwise, let $\lambda \in \N$ such that
      $\lambda v_m \equiv 1 \pmod{p}$. Set $\tilde{v}_i := \lambda v_i
      \mymod p$; then $\tilde{v}_i \in \{ 0, \dots, p-1 \}$ and
      $\tilde{v}_m \in \{ 0, 1 \}$, and we have $(\tilde{v}_1, \dots,
      \tilde{v}_m) + p \Z^m = \lambda v$ and $\lambda + p \Z \in
      (\Z/p\Z)^*$.  Since $p G = \{ (0, \dots, 0) + \Z^m \}$, we see
      that every non-trivial cyclic subgroup of $G$ is of order~$p$,
      and the previous discussion shows that every such subgroup is
      generated by at least one element in the set from the statement
      of the lemma. \qed
  \end{enuma}
\end{proof}
  
In fact, we can also write down a minimal such set $\tilde{S}$ for
$(\Z/p\Z)^m$ directly as
\[ 
\tilde{S} = \biggl\{ (v_1, \dots, v_i, 1, 0, \dots, 0) + p \Z^m
\;\biggm| \begin{matrix} i \in \{ 0, \dots, m - 1 \}, \hfill \\ 
v_1, \dots, v_i \in \{ 0, \dots, p-1 \} \end{matrix} \biggr\}. 
\] 
This shows that $\bigl|\tilde{S}\bigr| = 1 + p + p^2 + \dots + p^{m-1}
= \frac{p^m - 1}{p - 1}$. For our algorithm, we can restrict to a
subset of $p^{m-1}$ elements, since we also search the volume
\[ 
V := \sum_{i=1}^{n-1} [0, 1] v_i + [0, \tfrac{1}{B}] v_n, 
\]
where the $v_i$ are a basis of $\R^n.$ Then we only need the elements
of the form $(v_1, \dots, v_{m-1}, 1) + p \Z^m$.
  
These two results imply Proposition~\ref{mainpropoverview}. Moreover,
we combine them as sketched in Section~\ref{sec:overview} to obtain
our algorithm. The following corollary presents the preceding
material in a way which leads directly to the algorithm and its
correctness. It is also helpful to compare it with the sketch in
Figure~\ref{fig:algsketch}.
  
\begin{corollary}
  \label{latticemaxcorollary}
  Assume that $\Lambda' = \sum_{i=1}^n \Z v_i$, and let $B > 0$ be
  arbitrary. Let $p_1, \dots, p_t$ be all primes~$\le B$. For $i \in
  \{ 1, \dots, t \}$, set \[ \tilde{S}_i := \{ \tfrac{1}{p_i} (a_1 v_1
  + \dots + a_{n-1} v_{n-1} + v_n) \mid a_1, \dots, a_{n-1} \in \{ 0,
  \dots, p_i - 1 \} \}. \] Moreover, define the volume \[ V_B :=
  \sum_{i=1}^{n-1} [0, 1) v_i + [0, \tfrac{1}{B}) v_n. \] Then
  $\Lambda_K = \Lambda'$ if, and only if, $\Lambda_K \cap (V_B \cup
  \bigcup_{i=1}^t \tilde{S}_i) = \{ 0 \}$. Otherwise, any non-trivial
  element $v$ of $\Lambda_K \cap (V_B \cup \bigcup_{i=1}^t
  \tilde{S}_i)$ gives a lattice~$\Lambda'' := \Lambda' + \Z v$ with
  $\Lambda' \subsetneqq \Lambda'' \subseteq \Lambda_K$.
    
  Moreover, $\vol(V_B) = \frac{1}{B} \det \Lambda'$, $t =
  O(\frac{B}{\log B})$ and $\sum_{i=1}^t \bigl|\tilde{S}_i\bigr| =
  O(B^n / \log B)$.
\end{corollary}
  
\begin{proof}
  Clearly \[ \biggl( V_B \cup \bigcup_{i=1}^t \tilde{S}_i \biggr) \cap
  \Lambda_K = \{ 0 \}. \] Now assume that $\Lambda' \subsetneqq
  \Lambda$. Let $p$ be a prime dividing~$i_{\Lambda'}$ and define \[
  \tilde{S} := \biggl\{ \tfrac{1}{p} (a_1 v_1 + \dots + a_i v_i +
  v_{i+1}) \;\biggm| \begin{matrix} i \in \{ 0, \dots, n - 1 \},
  \hfill \\ a_1, \dots, a_i \in \{ 0, \dots, p - 1 \} \end{matrix}
  \biggr\}; \] by Proposition~\ref{finitegrouptrivialdetectlemma},
  $\tilde{S}$ must contain a non-trivial element of $\Lambda_K$. In
  case~$p > B$, we have $\tilde{S} \subseteq V_B$; and in case $p \le
  B$, say $p = p_i$, we have $\tilde{S} \subseteq \tilde{S}_i \cup
  V_B$.
    
  Now $\vol(V_B) = \frac{1}{B} \vol\bigl(\sum_{i=1}^n [0, 1) v_i\bigr)
  = \frac{1}{B} \det \Lambda'$. Moreover, by the Prime Number Theorem,
  $t = O(\frac{B}{\log B})$. Finally, $\bigl|\tilde{S}_i\bigr| =
  p_i^{n - 1} \le B^{n - 1}$, whence $\sum_{i=1}^t
  \bigl|\tilde{S}_i\bigr| \le t B^{n - 1} = O(\frac{B^n}{\log
  B})$. \qed
\end{proof}
  
We have seen how the idea sketched in the last section can be made
rigorous. It translates in a straightforward manner into an algorithm,
as we will see in Section~\ref{sec:thealgsect}. The only missing
pieces are how to search for elements in $V \cap \Lambda_K$, and how
to test whether some $v \in \R^n$ lies in $\Lambda_K$. We will
investigate this in the next section.

\section{Baby-Step Giant-Step Search and Existence Testing}
\label{sec:bsgssearch}
  
We will now investigate how to search for elements of $\Lambda_K$ in
the set $V = \sum_{i=1}^n [0, 1] v_i,$ where $v_1,
\dots, v_n$ is a basis of $\R^n$.  We assume that this basis is mostly
orthogonal, i.e. $\abs{\det(v_1, \dots, v_n)}^{-1} \prod_{i=1}^n
\norm{v_i} = O(1)$. This means that $\vol(V) = O(\prod_{i=1}^n
\norm{v_i})$. The algorithm requires $O(\vol(V)^{1/2}
\Delta_K^\epsilon)$~bit operations and was first described by Buchmann
in \cite{buchmann-habil}. We will also describe how to test whether an
element~$v \in \R^n$ lies in $\Lambda_K$.
  
For describing these algorithms, we need fractional ideals and the
notion of minima of these. A \emph{fractional ideal} is a finitely
generated $\calO_K$-submodule of $K$; it is always of the form
$\frac{1}{f} \fraka$, where $\fraka$ is an (integral) ideal of
$\calO_K$ in the usual sense and $f \in \calO_K \setminus \{ 0 \}$. As
$\calO_K$ is a Dedekind domain, the nonzero fractional ideals form a
free abelian group~$\Id(K)$ generated by the prime ideals of
$\calO_K$.
  
To define a minimum of an ideal, we use methods from Minkowski's
geometry of numbers. Set $W_i := \R$ if $\deg \abs{\bullet}_i = 1$ and
$W_i := \C$ otherwise. Then \[ \Phi : K \to W_K := \prod_{i=1}^{n+1}
W_i, \qquad f \mapsto (\sigma_1(f), \dots, \sigma_{n+1}(f)) \] is
injective and maps every fractional ideal~$\fraka \in \Id(K)$ onto a
lattice in the $[K : \Q]$-dimensional real vector space~$W_K \cong K
\otimes_\Q \R$. For $\fraka \in \Id(K)$ and $t_1, \dots, t_{n+1} \in
\R_{> 0}$, define \[ B(\fraka, t_1, \dots, t_{n+1}) := \{ f \in \fraka
\mid \abs{f}_i \le t_i \}. \] Then $\Phi$ identifies $B(\fraka, t_1,
\dots, t_{n+1})$ with the finite set of elements in $\Phi(\fraka)$ which lie
in the bounded area $\{ (v_1, \dots, v_{n+1}) \in W_K \mid \abs{v_i}
\le t_i \}$. For convenience, define
\begin{align*}
  B(\fraka, f, f') :={} & B(\fraka, \max\{\abs{f}_1, \abs{f'}_1 \}, \dots, \max\{\abs{f}_{n+1},
  \abs{f'}_{n+1}\} )
  \intertext{and} 
  B(\fraka, f) :={} & B(\fraka, f, f) = B(\fraka, \abs{f}_1, \dots,
  \abs{f}_{n+1})
\end{align*}
for $f, f' \in K^*$. Using this notation, we have that $B(\calO_K,
\varepsilon) = \{ 0 \} \cup \mu_K \varepsilon$ if $\varepsilon \in
\calO_K^*$.
  
Let $\fraka \in \Id(K)$. We say that $\mu \in \fraka$ is a
\emph{minimum} of $\fraka$ if $f \in B(\fraka, \mu) \setminus \{ 0 \}$
implies $\abs{f}_i = \abs{\mu}_i$ for some~$i$. Denote the set of all
minima of $\fraka$ by $\calE(\fraka)$. We say that $\fraka$ is
\emph{reduced} if $1 \in \calE(\fraka)$. Note that $\calO_K$ itself is
reduced.
  
The set $\Psi(\calE(\fraka))$ is distributed rather uniformly in
$\R^n$; here, $\Psi$ is as defined in Section~\ref{sec:notation}. More
precisely, Buchmann showed the following.
\begin{proposition}[\cite{buchmann-ontheperiodlength,buchmann-habil}]
  \label{distributionprop}
  Let $V = \sum_{i=1}^n [a_i, b_i] v_i$, where $(v_1, \dots, v_n)$ is
  a basis of $\R^n$ and $a_i < b_i$.
    \begin{enuma}
      \item Assuming that the $v_i$'s are mostly orthogonal, the set $\Psi(\calE(\fraka)) \cap V$
      contains $O(\vol(V))$~elements.
      \item If $V$ contains a sphere of radius~$\tfrac{1}{4} \sqrt{n} \log \Delta_K$, $V \cap
      \Psi(\calE(\fraka)) \neq \emptyset$.
    \end{enuma}
\end{proposition}
  
Note that $f \mu \in \calE(f \fraka)$ for $f \in K^*$ and $\mu \in
\calE(\fraka)$, and moreover that $1 \in \calE(\calO_K)$. This implies
that $\calO_K^*$ operates on $\calE(\fraka)$ and that $\calO_K^*
\subseteq \calE(\calO_K)$. It turns out that $\calE(\fraka) /
\calO_K^*$ is finite and contains~$O(R_K)$ elements
\cite[Theorem~2.1]{buchmann-ontheperiodlength}. Moreover, note that
the map $\calE(\fraka) / \calO_K^* \to \Id(K)$, $\mu \calO_K^* \mapsto
\frac{1}{\mu} \fraka$ is a bijection between $\calE(\fraka) /
\calO_K^*$ and the set of reduced ideals equivalent to
$\fraka$. Denote this set of ideals by $\Red(\fraka)$. This allows one
to represent an element~$\mu$ of $\calE(\fraka)$ up to a root of unity
by the pair $(\frac{1}{\mu} \fraka, \Psi(\mu))$. In practice, one
stores $\frac{1}{\mu} \fraka$ together with an approximation of
$\Psi(\mu)$.
  
The set of minima of an ideal modulo units is known as the
\emph{infrastructure} of that ideal. More precisely, consider the map
$\Psi$, together with the lattice~$\Lambda_K = \Psi(\calO_K^*)$; the
map $d^\fraka : \calE(\fraka) / \calO_K^* \to \R^n / \Lambda_K$, $\mu
\mapsto \Psi(\mu) + \Lambda_K,$ respectively $d^\fraka : \Red(\fraka)
\to \R^n / \Lambda_K$, $\frac{1}{\mu} \fraka \mapsto \Psi(\mu) +
\Lambda_K,$ is called the \emph{distance map}.

We now discuss on how to search for all minima~$\mu \in \calE(\fraka)$
with $\Psi(\mu) \in V$. For that, we need the notion of neighboring
minima as described in \cite{genlagrange}. Two minima~$\mu, \mu' \in
\calE(\fraka)$ are said to be \emph{neighbors} if $f \in B(\fraka,
\mu, \mu') \setminus \{ 0 \}$ implies $\abs{f}_i = \max\{ \abs{\mu}_i,
\abs{\mu'}_i \}$ for some~$i$. This relation defines a graph structure
on $\calE(\fraka)$ and $\calE(\fraka) / \calO_K^*$, and Buchmann
showed that this graph is connected \cite{genlagrange}. Moreover,
Buchmann showed that if $\fraka$ is a reduced ideal, one can compute
the set of all neighbors of $1 \in \calE(\fraka)$ in
$O(\Delta_K^\epsilon)$~bit operations; in fact, the number of
neighbors is in $O((\log \Delta_K)^n)$.
  
Using this, one can compute the set of all minima of $\fraka$ in $V$
in $O(\vol(V) \Delta_K^\epsilon)$~bit operations. Moreover, one can
test whether $\mu \calO_K^* = \mu' \calO_K^*$ by computing
$\frac{1}{\mu} \fraka$ and $\frac{1}{\mu'} \fraka$ and comparing
these. In fact, if one works with $(\frac{1}{\mu} \fraka, \Psi(\mu))$
instead of $\mu$ directly, one can do this easily by comparing the
ideals in the representations. Another reason to use this
representation of $\calE(\fraka)/\mu_K$ is that this representation is
small: Thiel showed that one can represent a reduced ideal with at
most $([K : \Q]^2 + 1) \log_2 \sqrt{\Delta_K}$ bits
\cite[Corollary~3.7]{thiel-comprep}. Hence, the storage required to
store all minima~$\mu \in \calE(\fraka)$ with $\Phi(\mu) \in V$ is
$O(\vol(V) \Delta_K^\epsilon)$ bits.
  
We can use this to employ a baby-step giant-step strategy similar to
the one in \cite{buchmann-habil} to search for elements in $V \cap
\Lambda_K$, where $V = \sum_{i=1}^n [0, 1] v_i$. Select integers~$a_1,
\dots, a_n > 0$ and set 
\[ 
B := -\sum_{i=1}^n v_i [0, \tfrac{1}{a_i}] + S 
\quad \text{and} \quad 
G := \biggl\{ \sum_{i=1}^n \tfrac{b_i}{a_i} v_i \;\biggm|\; b_i \in
\N, \, 0 \le b_i < a_i \biggr\},
\] 
where 
\[ 
S := \{ v \in \R^n \mid \norm{v} \le \tfrac{1}{4} \sqrt{n} \log
\Delta_K \}. 
\] 
The sets $B$ and $G$ are depicted in Figure~\ref{fig:bsgs-a}.
  
Let $\calE_B := \{ (\frac{1}{\mu} \fraka, \Psi(\mu)) \mid \mu \in
\calE(\fraka), \; \Psi(\mu) \in B \}$; this set is called the
\emph{baby stock}. For every~$v \in G$, one can find at least one $\mu
\in \calE(\fraka)$ with $\Psi(\mu) \in v + S$ by
Proposition~\ref{distributionprop}~(b); choose an arbitrary such $\mu$
as $\mu_v$ and set $\calE_G := \{ (\frac{1}{\mu_v} \fraka,
\Psi(\mu_v)) \mid v \in G \}$. Finding $\mu_v$ from $v$ is called a
\emph{giant step}. Using the strategy in Section~11 of
\cite{buchmann-habil}, $(\frac{1}{\mu_v} \calO_K, \Psi(\mu_v))$ can be
computed in $O(\log \norm{v} \cdot \Delta_K^\epsilon)$~bit
operations. The elements of $\calE_B$ and $\calE_G$ are depicted in
Figure~\ref{fig:bsgs-b}.
  
\begin{figure}[ht]
  \psfrag{B}{$B$}
  \psfrag{G}{$G$}
  \psfrag{v1}{$v_2$}
  \psfrag{v2}{$v_1$}
  \begin{center}
    \subfloat[The sets $B$ and $G$, where $a_1 = a_2 = 3$]%
             {\label{fig:bsgs-a}\includegraphics[width=3.5cm]{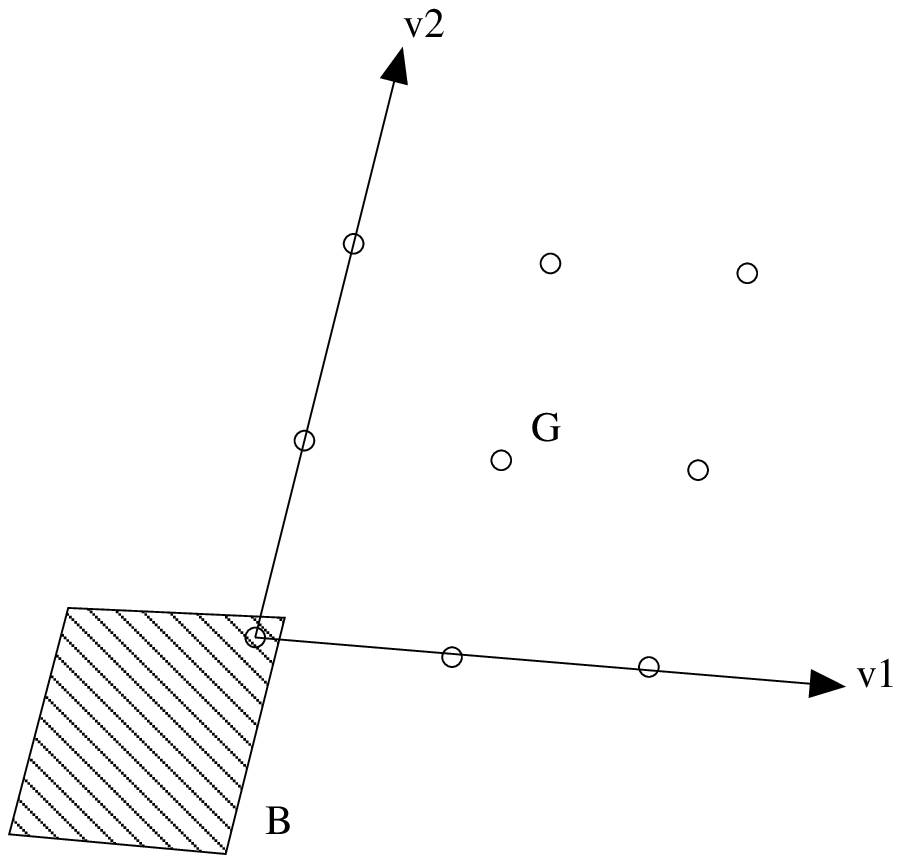}}
    \qquad\qquad
    \subfloat[The baby stock and the giant step minima]%
             {\label{fig:bsgs-b}\includegraphics[width=3.5cm]{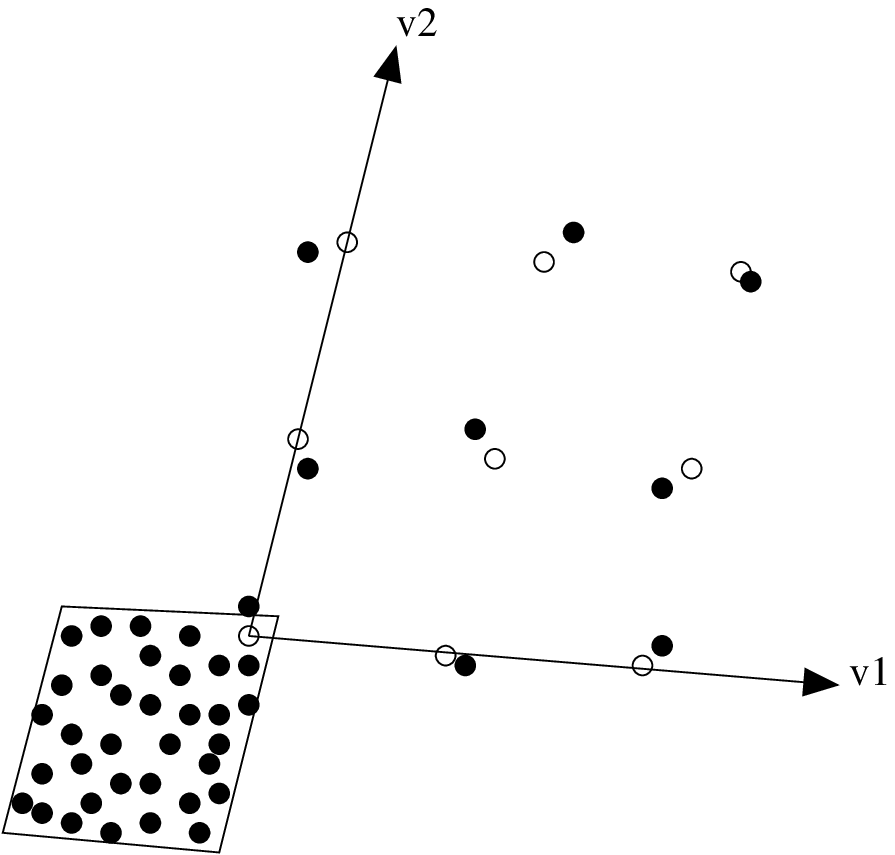}}
  \end{center}
  \caption{Visualization of the baby-step giant-step strategy}
  \label{fig:bsgs}
\end{figure}
  
\begin{proposition}
  \label{bsgsprop}
  For every~$\lambda \in V \cap \Lambda_K$, there exists an ideal
  $\fraka \in \Red(\calO_K)$ such that $(\fraka, v) \in \calE_B$,
  $(\fraka, w) \in \calE_G$ for some~$v, w \in \R^n$ such that
  $\lambda = w - v$.
    
  Conversely, given an ideal~$\fraka$ such that $(\fraka, v) \in
  \calE_B$, $(\fraka, w) \in \calE_G$ for some~$v, w \in \R^n$, then
  $w - v \in \Lambda_K$ with $w - v \in V + 2 S$.
\end{proposition}
  
\begin{proof}
  First, note that $\frac{1}{\mu} \fraka = \frac{1}{\mu'} \fraka$ if,
  and only if, $\mu^{-1} \mu' \in \calO_K^*$; therefore,
  $\frac{1}{\mu} \fraka = \frac{1}{\mu'} \fraka$ if, and only if,
  $\Psi(\mu') - \Psi(\mu) \in \Psi(\calO_K^*) = \Lambda_K$.
    
  Now if $\lambda \in V \cap \Lambda_K$, we can write $\lambda =
  \sum_{i=1}^n \lambda_i v_i$ with $\lambda_i \in [0, 1]$. Write
  $\lambda_i = \mu_i + \frac{b_i}{a_i}$ with $b_i \in \N$, $\mu_i \in
  [0, \frac{1}{a_i}]$. Set $w := \sum_{i=1}^n \frac{b_i}{a_i} v_i$;
  then $(\frac{1}{\mu_w} \fraka, \Psi(\mu_w)) \in \calE_G$ and
  $\hat{w} := \Psi(\mu_w) - w \in S$. Now $v := -\sum_{i=1}^n \mu_i
  v_i + \hat{w} \in B$; we have to show that $(\frac{1}{\mu_w} \fraka,
  v) \in \calE_B$, as $\Psi(\mu_w) - v = w + \hat{w} - v =
  \sum_{i=1}^n \frac{a_i}{b_i} v_i + \hat{w} + \sum_{i=1}^n \mu_i v_i
  - \hat{w} = \lambda$.
    
  For that, let $\varepsilon \in \calO_K^*$ with $\Psi(\varepsilon) =
  \lambda$. Now $\Psi(\varepsilon) = \lambda = \Psi(\mu_w) - v$,
  whence $v = \Psi(\mu_w \varepsilon^{-1})$. But $\frac{1}{\mu_w}
  \fraka = \frac{1}{\mu_w \varepsilon^{-1}} \fraka$, whence
  $(\frac{1}{\mu_w} \fraka, v) = (\frac{1}{\mu_w \varepsilon^{-1}}
  \fraka, \Psi(\mu_w \varepsilon^{-1})) \in \calE_B$. \qed
\end{proof}
  
Hence, to find all elements in $V \cap \Lambda_K$, one can enumerate
and store $\calE_B,$ enumerate all elements $v \in G$, compute a
corresponding $\mu_v$, and see if $(\frac{1}{\mu_v} \fraka, v) \in
\calE_B$ for some $v \in \R^n$. If that is the case, one obtains an
element of $\Lambda_K \cap (V + S)$, and the proposition shows that
every element of $\Lambda_K \cap V$ can be obtained in this way. As in
\cite{buchmann-habil}, this yields the following.
\begin{corollary}
  \label{bsgscorollary}
  Let $R = \vol(V)$. The strategy sketched above computes all elements
  in $V \cap \Lambda_K$ in $O( (R \prod_{i=1}^n n_i^{-1} +
  \prod_{i=1}^n n_i \cdot \log \max\limits_{i=1,\dots,n} \norm{v_i} )
  \Delta_K^\epsilon)$~bit operations and requires $O( R \prod_{i=1}^n
  n_i^{-1} \cdot \Delta_K^\epsilon)$~bits of storage.
\end{corollary}
Note that the running time is minimized if $\prod_{i=1}^n n_i \approx \sqrt{R}$.
  
\begin{proof}
  The storage requirements follow from
  Proposition~\ref{distributionprop}~(a) and
  \cite[Corollary~3.7]{thiel-comprep}. Using the enumeration technique
  by Buchmann \cite{genlagrange,buchmann-habil}, one can compute
  $\calE_B$ in $O( R \prod_{i=1}^n n_i^{-1} \cdot \Delta_K^\epsilon)$~bit
  operations since $\vol(B) = O(R \cdot \prod_{i=1}^n n_i^{-1} \cdot
  \Delta_K^\epsilon)$. Finally, one can compute the elements in
  $\calE_G$ in $O(\abs{\calE_G} \Delta_K^\epsilon \cdot \log \max
  \norm{v_i})$~bit operations. \qed
\end{proof}
  
Finally, we discuss how to test whether $v \in \Lambda_K$ for some $v
\in \R^n$. We use the giant step strategy mentioned above to compute
some $\mu \in \calE(\calO_K)$ with $\Psi(\mu) \in v + S$. Then, one
uses the above strategy to enumerate all minima $\mu' \in
\calE(\frac{1}{\mu} \calO_K)$ with $\Psi(\mu') \in S$ to check whether
a minimum~$\mu'$ with $\Psi(\mu') + \Psi(\mu) = v$ and $\frac{1}{\mu'}
(\frac{1}{\mu} \calO_K) = \calO_K$ exists.
\begin{lemma}
  Let $\mu \in \calE(\calO_K)$ and $v \in \R^n$. Then there exists a
  minimum~$\mu' \in \calE(\frac{1}{\mu} \calO_K)$ with $\Psi(\mu') +
  \Psi(\mu) = v$ such that $\frac{1}{\mu' \mu} \calO_K = \calO_K$ if,
  and only if, $v \in \Lambda_K$.
\end{lemma}
  
\begin{proof}
  First, assume that $\Psi(\mu') + \Psi(\mu) = v$ and $\frac{1}{\mu'
  \mu} \calO_K = \calO_K$. Then $\mu \mu' \in \calO_K^*$ and $v =
  \Psi(\mu \mu') \in \Lambda_K$. Conversely, assume that $v \in
  \Lambda_K$, say $v = \Psi(\varepsilon)$ with $\varepsilon \in
  \calO_K^*$. But $\varepsilon \in \calE(\calO_K)$ and $\mu' :=
  \frac{\varepsilon}{\mu} \in \calE(\frac{1}{\mu} \calO_K)$, and
  $\Psi(\mu) + \Psi(\mu') = \Psi(\varepsilon) = v.$\qed
\end{proof}
  
Note that one can compute $\mu_v$ in $O(\log \norm{v} \cdot
\Delta_K^\epsilon)$~bit operations, and $S \cap \Psi(\frac{1}{\mu_v}
\calO_K)$ contains $O(\Delta_K^\epsilon)$ elements by
Proposition~\ref{distributionprop}~(a). Hence we obtain the following corollary.
\begin{corollary}
  \label{testingcorr}
  Given $v \in \R^n$, one can test whether $v \in \Lambda_K$ in
  $O(\log \norm{v} \cdot \Delta_K^\epsilon)$~bit operations and
  $O(\Delta_K^\epsilon)$~bits of storage. \qed
\end{corollary}
  
We have seen how we can deploy a baby-step giant-step strategy to
search for elements in $V \cap \Lambda_K$. Moreover, we saw how to
test whether a given~$v \in \R^n$ is an element of $\Lambda_K$.  These
two methods are the required computational tools to translate the
lattice maximization strategy of Corollary~\ref{latticemaxcorollary}
into an algorithm.

\section{The Algorithm}
\label{sec:thealgsect}
  
The algorithm is in a rather straightforward way based on
Corollary~\ref{latticemaxcorollary} combined with a baby-step
giant-step strategy as outlined in Section~\ref{sec:bsgssearch}. It is
formalized in Algorithm~\ref{latticemaximizationalg}.  The correctness
of this algorithm follows directly from
Corollaries~\ref{latticemaxcorollary} and \ref{bsgscorollary}.
    
  \begin{algorithm}[ht]
    \caption{Find~$\Lambda_K \subset \R^n$, given a sublattice~$\Lambda'$ of full rank.}
    \label{latticemaximizationalg}
    \begin{algorithmic}[1]
      \REQUIRE A basis $(v_1, \dots, v_n)$ of $\Lambda' \subseteq \Lambda_K$, a parameter~$B \ge 1$,
      a parameter~$\delta \in (0, 1)$.
      \ENSURE A basis of $\Lambda_K$. 
      \STATE\label{alg1:step1} Reduce the basis $(v_1, \dots, v_n)$, i.e. make it mostly orthogonal.
      \FOR{all primes $p$ with $2 \le p \le B$}\label{alg1:step2}
        \FOR{all $(a_1, \dots, a_{n-1}) \in \{ 0, \dots, p-1 \}^{n-1}$}\label{alg1:step2a}
          \STATE Set $v = \frac{1}{p} (a_1 v_1 + \dots + a_{n-1} v_{n-1} + v_n)$.
          \IF{$v \in \Lambda_K$}
            \STATE Compute a reduced basis $(\hat{v}_1, \dots, \hat{v}_n)$ of $\ggen{ v_1, \dots,
            v_n, v }_\Z$.
            \STATE Replace $(v_1, \dots, v_n)$ by $(\hat{v}_1, \dots, \hat{v}_n)$ and restart the
            loop in line~\ref{alg1:step2a}.
          \ENDIF
        \ENDFOR\label{alg1:step2aend}
      \ENDFOR\label{alg1:step2end}
      \STATE\label{alg1:step3} Determine $a_1, \dots, a_n \in \N_{>0}$ such that $\prod_{i=1}^n a_i \approx
      (\frac{1}{B} \det \Lambda')^{1 - \delta}$.
      \FOR[$S$ as in Section~\ref{sec:bsgssearch}]{all $\mu \in \calE(\calO_K)$ with $\Psi(\mu)
      \in \sum_{i=1}^n v_i [-\tfrac{1}{a_i}, 0] + S$}\label{alg1:step4}
        \STATE Store $(\frac{1}{\mu} \calO_K, \Psi(\mu))$ in the set $\calE_B$.
        \IF{some $(\calO_K, v) \in \calE_B$ with $v \not\in \ggen{v_1, \dots, v_n}_\Z$ is found}
          \STATE Compute a reduced basis $(\hat{v}_1, \dots, \hat{v}_n)$ of $\ggen{ v_1, \dots, v_n,
          v }_\Z$.
          \STATE\label{alg1:step4restart} Replace $(v_1, \dots, v_n)$ by $(\hat{v}_1, \dots,
          \hat{v}_n)$ and go back to line~\ref{alg1:step3}.
        \ENDIF
      \ENDFOR\label{alg1:step4end}
      \FOR{all $w \in \{ \sum_{i=1}^n \frac{a_i}{b_i} v_i \mid a_i \in \N, \; 0 \le a_i < b_i
      \}$}\label{alg1:step5}
        \STATE Compute some $(\frac{1}{\mu} \calO_K, \Psi(\mu))$ with $\mu \in
        \calE(\calO_K)$ and $\Psi(\mu) \in w + S$.
        \IF{$(\frac{1}{\mu} \calO_K, v)$ is found in $\calE_B$ for some $v \in \R^n$ with $\Psi(\mu)
        - v \not\in \ggen{v_1, \dots, v_n}_\Z$}
          \STATE Compute a reduced basis $(\hat{v}_1, \dots, \hat{v}_n)$ of $\ggen{ v_1, \dots, v_n,
          \Psi(\mu) - v }_\Z$.
          \STATE Replace $(v_1, \dots, v_n)$ by $(\hat{v}_1, \dots, \hat{v}_n)$ and go back to
          line~\ref{alg1:step3}.
        \ENDIF
      \ENDFOR\label{alg1:step5end}
      \RETURN $(v_1, \dots, v_n)$.
    \end{algorithmic}
  \end{algorithm}

During the course of the algorithm, we try to keep the basis
vectors~$v_1, \dots, v_n$ as orthogonal as possible; in that case, we
have $\abs{\det(v_1, \dots, v_n)} \approx \prod_{i=1}^n
\norm{v_i}$. Such a basis can be computed as in Algorithm~16.10 of
\cite{vzgathen-moderncomputeralgebra} and is called a \emph{reduced}
basis. 
  
We now analyze the asymptotic running time and memory consumption of
Algorithm~\ref{latticemaximizationalg}. Recall that $[K : \Q] = O(1)$;
note that the $O$-constants are assumed to be exponentially dependent
on $n$ (compare \cite[p.~5]{buchmann-habil}).

\begin{theorem}
  \label{algcomplexthm}
  Algorithm~\ref{latticemaximizationalg} requires \[ O\Bigl( \bigl(
  (\tfrac{1}{B} \det \Lambda')^\delta + (\tfrac{1}{B} \det
  \Lambda')^{1 - \delta} + B^n (\log B)^{-1} \bigr) (\Delta_K \det
  \Lambda')^\epsilon \Bigr) \] bit operations and $O( (\tfrac{1}{B}
  \det \Lambda')^\delta \Delta_K^\epsilon)$ bits of storage.
\end{theorem}
  
\begin{proof}
  First, assume that $\Lambda' = \Lambda_K$, i.e. no element in
  $\Lambda_K \setminus \Lambda'$ is found $\Lambda'$ is not replaced
  by a larger sublattice of $\Lambda_K.$
    
  The loop in lines~\ref{alg1:step2}--\ref{alg1:step2end} requires
  $O(\frac{B^n}{\log B} (\det \Lambda')^\epsilon
  \Delta_K^\epsilon)$~bit operations as well as
  $O(\Delta_K^\epsilon)$~bits of storage by the
  Corollaries~\ref{latticemaxcorollary} and \ref{testingcorr}.  Note
  that the primes required can be computed with the Sieve of
  Eratosthenes in time $O(B^{1 + \epsilon})$ bit operations, so this
  part of the computation does not affect the overall asymptotic
  running time.
    
  The value~$R$ of Corollary~\ref{bsgscorollary} is in $O(\frac{1}{B} \det \Lambda')$ by
  Corollary~\ref{latticemaxcorollary}. Hence, by Corollary~\ref{bsgscorollary}, the loops in
  lines~\ref{alg1:step4}--\ref{alg1:step4end} and \ref{alg1:step5}--\ref{alg1:step5end} require $O(
  (\frac{1}{B} \det \Lambda' \cdot (\frac{1}{B} \det \Lambda')^{-(1 - \delta)} + (\frac{1}{B} \det
  \Lambda')^{1 - \delta}) \Delta_K^\epsilon) = O( ((\frac{1}{B} \det \Lambda')^\delta + (\frac{1}{B}
  \det \Lambda')^{1 - \delta}) \Delta_K^\epsilon)$~bit operations and \\ $O( (\frac{1}{B} \det
  \Lambda')^\delta \Delta_K^\epsilon)$~bits of storage.
    
  Now, every time one finds an element in $\Lambda_K \setminus
  \Lambda'$, the index~$[\Lambda' : \Lambda_K]$ and $\det \Lambda'$
  are divided by at least two. Hence, $\Lambda'$ is replaced at most
  $\log_2 [\Lambda' : \Lambda_K]$ times. Now $[\Lambda' : \Lambda_K] =
  O(\det \Lambda')$; therefore, the above bounds for the number of bit
  operations needs to be multiplied by $\log_2 \det \Lambda' = O(
  (\det \Lambda')^\epsilon )$.
    
  Note that we can ignore the running time for the orthogonalization
  process.  By Theorem~16.11 in \cite{vzgathen-moderncomputeralgebra},
  the running time of the basis reduction algorithm is bounded by
  $O(n^4 \log A)$~arithmetic operations on integers of length~$O(n
  \log A)$, where $A = \max\{ \norm{v_1}, \dots, \norm{v_n} \}$. Since
  $n = O(1)$ in our notation, the running time is bounded by $O((\det
  \Lambda')^\epsilon)$~bit operations. \qed
\end{proof}
  
We now optimize the running time for two situations. For our
optimizations, we simplify the upper bound from
Theorem~\ref{algcomplexthm} by omitting the $(\log B)^{-1}$ factor;
then the running time is bounded by \[ O\Bigl( \bigl( (\tfrac{1}{B}
\det \Lambda')^\delta + (\tfrac{1}{B} \det \Lambda')^{1 - \delta} +
B^n \bigr) (\Delta_K \det \Lambda')^\epsilon \Bigr) \] bit
operations. Moreover, we ignore the $(\Delta_K \det
\Lambda')^\epsilon$ part, i.e., we assume that all three operations
(existence testing, baby stock computation, giant steps) are equally
fast. Hence, we need to minimize the term $(\tfrac{1}{B} \det
\Lambda')^\delta + (\tfrac{1}{B} \det \Lambda')^{1 - \delta} + B^n$.
  
Note that these two simplifications are justified.  If we minimize the
original formula, the difference to our minimal running time can be
bounded by $O((\Delta_K \det \Lambda')^\epsilon)$, i.e. can be ignored
since we have the factor $O((\Delta_K \det \Lambda')^\epsilon$)
anyway.
  
First, we optimize without any restrictions on the amount of available memory.
\begin{corollary}
  \label{optimalperformance}
  If $B$ and $\delta$ can be chosen freely, optimal performance of
  Algorithm~\ref{latticemaximizationalg} is obtained for $\delta =
  \tfrac{1}{2}$ and $B = (\det \Lambda')^{\frac{1}{2 n + 1}}
  n^{-\frac{2}{2 n + 1}}$. In that case, one needs $O( (\det
  \Lambda')^{\frac{n}{2 n + 1} + \epsilon} \Delta_K^\epsilon)$ bit
  operations and $O( (\det \Lambda')^{\frac{n}{2 n + 1}}
  \Delta_K^\epsilon )$ bits of storage.
\end{corollary}
  
\begin{proof}
  For fixed $B$, the expression $(\tfrac{1}{B} \det \Lambda')^\delta +
  (\tfrac{1}{B} \det \Lambda')^{1 - \delta} + B^n$ is minimal for
  $\delta = \frac{1}{2}$; in that case, it attains the value $2
  B^{-1/2} \sqrt{\det \Lambda'} + B^n$.
    
  Differentiating this by $B$, we obtain $-\sqrt{\det \Lambda'}
  B^{-3/2} + n B^{n - 1}$. This is zero if, and only if, $B = (\det
  \Lambda')^{\frac{1}{2 n + 1}} n^{-\frac{2}{2 n + 1}}$. In that case,
  it attains the value $2 (\det \Lambda')^{\frac{n}{2 n + 1}}
  n^{\frac{1}{2 n + 1}} + (\det \Lambda')^{\frac{n}{2 n + 1}}
  n^{-\frac{2 n}{2 n + 1}}$. Plugging these choices for $\delta$ and
  $B$ in gives the result. \qed
\end{proof}
  
Next, we investigate the situation in which the available memory is
insufficient to store the optimal number of baby steps.
\begin{corollary}
  \label{memlimitperformance}
  Assume that storage is limited to $T$ baby steps, and that one has
  less memory than required for the optimal running time of
  Algorithm~\ref{latticemaximizationalg} as in
  Corollary~\ref{optimalperformance}.  Under this assumption, optimal
  performance of Algorithm~\ref{latticemaximizationalg} is obtained
  for $\delta = \frac{(1 + n) \log T}{\log T + n \log \det \Lambda'}$
  and $B = (\det \Lambda' / T)^{\frac{1}{n + 1}}$. In that case, one
  needs $O\bigl( \bigl( T + (\det \Lambda' / T)^{\frac{n}{n + 1} +
  \epsilon} \bigr) \Delta_K^\epsilon\bigr) = O\bigl( (\det \Lambda' /
  T)^{\frac{n}{n + 1} + \epsilon} \Delta_K^\epsilon \bigr)$ bit
  operations.
\end{corollary}
  
\begin{proof}
  In this case, the number of operations required for the ``baby
  steps'' in the loop in lines~\ref{alg1:step4}--\ref{alg1:step4end}
  of the algorithm is $O(T \Delta_K^\epsilon)$. As optimal performance
  as in Corollary~\ref{optimalperformance} can not be obtained, one
  needs to balance the number of operations for the loop in
  lines~\ref{alg1:step2}--\ref{alg1:step2end} and the one in
  lines~\ref{alg1:step5}--\ref{alg1:step5end}, i.e. one needs to
  choose~$\delta$ and $B$ such that $(\frac{1}{B} \det
  \Lambda')^\delta \approx T$ and $(\frac{1}{B} \det \Lambda')^{1 -
  \delta} \approx B^n (\log B)^{-1}$. For simplicity, we ignore the
  factor of $\frac{1}{\log B}$ as in
  Corollary~\ref{optimalperformance} and replace ``$\approx$'' by
  ``$=$''.
    
  The first equality gives $B = T^{-1/\delta} \det \Lambda'$, whence
  the second translates to $T^{\frac{1 - \delta}{\delta}} =
  T^{-n/\delta} (\det \Lambda')^n$. But this gives $(T^{1 +
  n})^{\frac{1}{\delta}} = T (\det \Lambda')^n$, i.e. $\delta =
  \frac{(1 + n) \log T}{\log T + n \log \det \Lambda'}$ and, hence, $B
  = (\det \Lambda' / T)^{\frac{1}{n + 1}}$. Plugging this in, we
  obtain the given bound. \qed
\end{proof}

\section{Conclusions}
  
We have seen that our algorithm computes $\Lambda_K$ in \[ O((\det
\Lambda')^{\frac{n}{2n + 1} + \epsilon} \Delta_K^\epsilon) = O((\det
\Lambda')^{1/2 - \frac{1}{4n + 2} + \epsilon} \Delta_K^\epsilon) \]
bit operations, using $O((\det \Lambda')^{\frac{n}{2 n + 1}}
\Delta_K^\epsilon)$~bits of storage.  In particular, our algorithm
generalizes the algorithm in
\cite{haan-jacobson-williams-fastrigorous} to number fields of
arbitrary unit rank, with the same complexity as
\cite{haan-jacobson-williams-fastrigorous} being obtained in our
algorithm for unit rank $1.$ In the case that memory is too limited
for the optimal method, we determined for the value of $B$ for which
optimal performance is obtained when using a restricted amount of
memory.

If $\det \Lambda' = O(\Delta_K^{1/2 + \epsilon})$, for example when
$\Lambda'$ is computed using Buchmann's index-calculus algorithm and
is correct assuming the GRH, we obtain a complexity of
$O(\Delta_K^{1/4 - \frac{1}{8 n + 4} + \epsilon})$~bit
operations. Thus, computing $\Lambda'$ with Buchmann's algorithm
followed by our's to verify that $\Lambda' = \Lambda_K$ yields an
algorithm that computes $\Lambda_K$ unconditionally with expected
complexity $O(\Delta_K^{1/4 - \frac{1}{8n + 4} + \epsilon})$~bit
operations.  Only the complexity is dependent on the GRH, for both the
running time and correctness (required to bound the size of $\det
\Lambda'$) of Buchmann's algorithm.  This is always asymptotically
better than Buchmann's baby-step giant-step method for
computing~$\Lambda_K$, whose running time is $O(\Delta_K^{1/4 +
\epsilon})$~ bit operations. For unit rank~one, i.e. for $n = 1$, we
obtain~$O(\Delta_K^{1/6 + \epsilon})$~bit operations; this is the same
complexity as in \cite{haan-jacobson-williams-fastrigorous}. For unit
rank~two, we obtain $O(\Delta_K^{1/5 + \epsilon})$~bit operations;
this is faster than any other known algorithm for computing the units
of a number field of unit rank~two whose correctness of the output
does not depend on the GRH.

Even though the baby stock computation, giant step computation and
existence testing of lattice elements roughly need
$O(\Delta_K^\epsilon)$~bit operations, with some factor polynomial in
the logarithms of the dimensions of the involved objects, the running
times of these three operations vary a lot in practice. In particular,
computing all neighbors of a minimum is very slow compared to reducing
an ideal, which is the main operation when computing giant
steps. Therefore, in practice, it makes sense to first sample the
running times of these three operations, and to find optimal values of
$\delta$ and $B$ that take this into account in a manner similar to
the algorithm in \cite{haan-jacobson-williams-fastrigorous}.
Moreover, it is also possible re-adjust $\delta$ and $B$ after an
element in~$\Lambda_K \setminus \Lambda'$ is found, as this
changes~$\det \Lambda'$. One can also optimize the running time by
reusing the already computed part of $\calE_B$ when updating
$\Lambda'$ in line~\ref{alg1:step4restart}.

Another possible practical improvement is to parallelize parts of the
algorithm.  In particular, the loops in
lines~\ref{alg1:step2a}--\ref{alg1:step2aend} and
\ref{alg1:step2}--\ref{alg1:step2end} can easily be parallelized. The
loops in lines~\ref{alg1:step4}--\ref{alg1:step4end} and
\ref{alg1:step5}--\ref{alg1:step5end} can be parallelized in a similar
manner to all baby-step giant-step type algorithms.  As in
\cite{haan-jacobson-williams-fastrigorous}, it is possible to
re-optimize the running time to find optimal values of $\delta$ and $B$
that take into account parallelization and the number of processors
used.
 
Note that these optimizations do not affect the asymptotic complexity
of our algorithm.  However, as in the case of real quadratic fields
\cite{haan-jacobson-williams-fastrigorous}, we expect that they will
have a significant impact on its practical performance.

So far, we do not have an implementation of our algorithm. The main
problem is that the methods in Section~\ref{sec:bsgssearch}, or more
precisely computing all neighbors of $1$ in a reduced ideal, are not
implemented in any number theory library to our knowledge. All
libraries and computer algebra systems which provide methods for
computing units of number fields use Buchmann's subexponential
algorithm \cite{buchmann-subexp-regulator}. An implementation is not
yet available, but is currently work in progress. It will be
interesting to see how our algorithm performs in practice.

\end{document}